\documentclass[11pt,a4paper]{amsart}
\usepackage{amssymb,amsmath,epsfig,graphics,mathrsfs}

\usepackage{fancyhdr}
\usepackage{hyperref}
\hypersetup{
 colorlinks   = true,
 urlcolor     = blue,
 linkcolor    = blue,
 citecolor   = red ,
 bookmarksopen=true
}

\usepackage{amsmath}
\usepackage{amsfonts}
\usepackage{amssymb}
\usepackage{amsthm}
\usepackage{epsfig,graphics,mathrsfs}
\usepackage{graphicx}

\usepackage[usenames, dvipsnames]{color} 

\usepackage{hyperref}

\def \de {\partial}

\def \N {\mathbb{N}}

\def \phi {\varphi}

\def \RN {\mathbb{R}^N}
\def \R {\mathbb{R}}

\def \G{\Gamma}

\def \So {\mathscr{S}}
\newcommand{\sA}{\mathscr A}

\newcommand{\la}{\lambda}

\numberwithin{equation}{section}

\newcommand{\beq}{\begin{equation}}
\newcommand{\bea}[1]{\begin{array}{#1} }
\newcommand{\eeq}{ \end{equation}}
\newcommand{\ea}{ \end{array}}

\newcommand{\Lo}{\mathscr L^{2s,p}}
\newcommand{\As}{(-\mathscr A)^s}
\newcommand{\A}{\mathscr A}
\newcommand{\Ma}{\mathscr M}

\newcommand{\In}{\mathbf 1_E}

\newtheorem{theorem}{Theorem}[section]
\newtheorem{lemma}[theorem]{Lemma}
\newtheorem{proposition}[theorem]{Proposition}

\newtheorem{remark}[theorem]{Remark}
\newtheorem{definition}[theorem]{Definition}

\numberwithin{equation}{section}

\begin{document}

\title[A new proof of the geometric Sobolev embedding etc.]{A new proof of the geometric Sobolev embedding for generalised Kolmogorov operators}



\begin{abstract}
In this note we revisit a result in \cite{GTjfa}, where we established nonlocal isoperimetric inequalities and the related embeddings for Besov spaces adapted to a class of H\"ormander operators of Kolmogorov-type. We provide here a new proof which exploits a weak-type Sobolev embedding established in \cite{GThls}.
\end{abstract}

\author{Nicola Garofalo}

\address{Dipartimento d'Ingegneria Civile e Ambientale (DICEA)\\ Universit\`a di Padova\\ Via Marzolo, 9 - 35131 Padova,  Italy}
\vskip 0.2in
\email{nicola.garofalo@unipd.it}

\author{Giulio Tralli}
\address{Dipartimento d'Ingegneria Civile e Ambientale (DICEA)\\ Universit\`a di Padova\\ Via Marzolo, 9 - 35131 Padova,  Italy}
\vskip 0.2in
\email{giulio.tralli@unipd.it}

\maketitle


\section{Introduction}\label{S:intro}

Consider the Kolmogorov-Fokker-Planck operators in $\RN$, $N\geq 2$, defined as follows
\begin{equation}\label{A0}
\mathscr A u = \operatorname{tr}(Q \nabla^2 u) + \left\langle BX,\nabla u\right\rangle,
\end{equation}
where the $N\times N$ matrices $Q$ and $B$ have real, constant coefficients, $Q = Q^\star \ge 0$, and $X$ stands for the generic point in $\R^N$. The operators $\A$ in \eqref{A0} were introduced in \cite{Ho}, where H\"ormander showed that they are hypoelliptic if and only if the covariance matrix
\begin{equation}\label{Kt}
K(t) = \frac 1t \int_0^t e^{sB} Q e^{s B^\star} ds \quad\mbox{ is positive definite for every $t>0$.}
\end{equation}
This condition will be henceforth assumed throughout the whole paper. Equations such as \eqref{A0} are of considerable interest in physics, probability and finance, and have been the subject of intense study during the past three decades (see the recent survey article \cite{AP}, and the references therein). Besides the classical Laplace equation (which corresponds to the non-degenerate model $Q = I_N$, $B = O_N$), they encompass the Ornstein-Uhlenbeck operator (which is obtained by taking $Q = I_N$ and $B = - I_N$), as well as the degenerate operator of Kolmogorov in $\R^{2n}$
\begin{equation}\label{kolmo}
\A_0 u = \Delta_v u + <v,\nabla_x u >,
\end{equation}
corresponding to the choice $N=2n$, $Q = \begin{pmatrix} I_n & 0_n\\ 0_n& 0_n\end{pmatrix}$, and $B = \begin{pmatrix} 0_n & 0_n\\ I_n & 0_n\end{pmatrix}$. Such operator arises in the kinetic theory of gases and was first introduced in the seminal note \cite{Ko} on Brownian motion. One should note that $\A_0$ fails to be elliptic since it is missing the diffusive term $\Delta_x u$. However, it does satisfy H\"ormander's hypoellipticity condition since one easily checks that $K(t)=\begin{pmatrix} I_n & t/2\ I_n\\ t/2\ I_n& t^2/3\ I_n\end{pmatrix}>0$ for every $t>0$. In this respect, it should be noted that Kolmogorov himself had already shown the hypoellipticity of his operator since in \cite{Ko} he constructed an explicit heat kernel for $\A_0$ which is $C^\infty$ outside the diagonal. Kolmogorov's construction was generalised in \cite{Ho}, where it was shown that the heat kernel for $\A$ can be explicitly written as
\begin{equation}\label{PtKt}
p(X,Y,t) = \frac{\omega_N (4\pi)^{-\frac N2}}{V(t)} \exp\left( - \frac{m_t(X,Y)^2}{4t}\right).
\end{equation}
In \eqref{PtKt}, for $X, Y\in \RN$ we have let, for $t>0$,
\begin{align}\label{defs}
m_t(X,Y) & = \sqrt{\left\langle K(t)^{-1}(Y-e^{tB} X ),Y-e^{tB} X \right\rangle},\nonumber\\
B_t(X,r) &= \{Y\in \RN\mid m_t(X,Y)<r\}\nonumber\\
V(t) &= \left|B_t(X,\sqrt t)\right| = \omega_N  (\det(t K(t)))^{1/2},
\end{align}
where $\omega_N$ indicates the Lebesgue measure $\left|\hphantom{l}\cdot\hphantom{l}\right|$ of the unit ball in $\RN$. If we indicate with 
$$
P_t f(X) = \int_{\RN} p(X,Y,t) f(Y) dY
$$
 the H\"ormander semigroup, then it is well-known that
$$
P_t 1\equiv 1\qquad\mbox{ and }\qquad P^*_t 1\equiv e^{-t \operatorname{tr} B}.
$$
As a consequence, under the assumption that the matrix $B$ of the drift satisfies 
\begin{equation}\label{trace}
\operatorname{tr} B \ge 0,
\end{equation} 
we obtain a non-symmetric semigroup which is contractive in $L^p:=L^p\left(\RN, dX\right)$, $1\le p\le \infty$. Under the condition \eqref{trace}, in a series of papers \cite{GTind, GThls, GTfi, GTjfa, BGT} we have developed some basic functional analytic aspects of the class \eqref{A0}. In particular, since the operators $\mathscr A$ do not possess a variational structure, one focus of our investigation has been a notion of \emph{gradient}. Guiding by the idea that the heat kernel \eqref{PtKt} should encapsulate the geometry underlying $\mathscr A$, we introduced the following class of Besov spaces naturally associated with the semigroup $P_t$.

\begin{definition}
For any $0<s<1$ and $1\le p<\infty$, the Besov space $\mathfrak B^\sA_{s,p}$ is the collection of all functions $f\in L^p$ such that
\begin{equation}\label{sn}
\mathscr N^\sA_{s,p}(f) = \left(\int_0^\infty  \frac{1}{t^{\frac{s p}2 +1}} \int_{\RN} P_t\left(|f - f(X)|^p\right)(X) dX dt\right)^{\frac 1p} < \infty.
\end{equation}
\end{definition}
Under condition \eqref{trace}, we know that smooth functions with compact support are contained and dense in $\mathfrak B^\sA_{s,p}$ (see Proposition \ref{dens} below). Denoting by $\In$ the indicator function of a set $E$, we have introduced the following.
\begin{definition}\label{defper}
Fix $0<s<\frac{1}{2}$. We say that a measurable set $E\subset \R^N$ has finite $s$-\emph{perimeter} if $\In\in \mathfrak B^\sA_{2s,1}$ and we define the $s$-perimeter associated to $\sA$ as
\[
\mathfrak P_{\sA,s}(E) = \mathscr N^\sA_{2s,1}(\In).
\]
\end{definition}
The reader should keep in mind that 
\begin{equation}\label{weighthejump}
\mathscr N^\sA_{2s,1}(\In)=\int_0^{\infty} \frac{1}{t^{1+s}}\|P_t\In-\In\|_1 dt,
\end{equation} 
see in this respect \cite[Corollary 3.6]{GTfi} and \cite[Section 3]{GTjfa}. Formula \eqref{weighthejump} underscores the role of the $1$-parameter family of pseudo-distances $m_t(\cdot,\cdot)$ in \eqref{PtKt} via the $s$-dependent average of $\|P_t\In-\In\|_1$.

When specialized to $\A=\Delta$, up to a renormalizing factor the seminorm $\mathscr N^\sA_{s,p}(\cdot)$ coincides  with the classical Aronszajn-Gagliardo-Slobedetzky seminorm: in such framework it is nowadays a common practice to call \emph{nonlocal perimeter} the Aronszajn-Gagliardo-Slobedetzky seminorm of the indicator function, and we refer the reader to the influential work \cite{CRS} where the structure of the critical points of nonlocal perimeters was first analyzed. The notion of nonlocal perimeter in the Euclidean setting was implicitly present in \cite{AL, Ma1}, and the nonlocal counterpart of the classical De Giorgi's isoperimetric inequality \cite{DGper54} was established in \cite{AL, FS} with respect to such notion.\\
Concerning the general class of operators $\A$ in \eqref{A0}, we established in \cite[Theorem 1.1 and Theorem 1.3]{GTjfa} the following nonlocal isoperimetric inequalities and the related embeddings of the Besov spaces $\mathfrak B^{\A}_{2s,1}$ in the relevant $L^q$-space of functions with higher integrability properties.

\begin{theorem}\label{T:alt1}
Let $0<s< \frac{1}{2}$, and assume \eqref{trace}. Suppose there exist $D, \gamma_D>0$ such that
\begin{equation}\label{vol}
V(t) \geq \gamma_D\ t^{D/2}\quad\mbox{ for all }t>0.
\end{equation}
Then
$$\mathfrak B^{\A}_{2s,1}\hookrightarrow L^{\frac{D}{D-2s}}.$$ 
More precisely, there exists a positive constant $c$ (depending on $N,D,s,\gamma_D$) such that for every $f\in \mathfrak B^{\A}_{2s,1}$ one has
\begin{equation}\label{besunoembe}
||f||_{\frac{D}{D-2s}} \leq \frac{1}{c} \mathscr N^{\A}_{2s,1}(f).
\end{equation}
In particular, for any measurable set $E\subset \RN$ with $|E|<\infty$, one has
$$
\mathfrak P_{\A, s}(E) \geq\ c\ |E|^{(D-2s)/D}.
$$
\end{theorem}

In Theorem \ref{T:alt1} the assumption \eqref{vol} on the growth of the volume function $V(t)$ is crucial as it allows to detect the dimensional parameter $q=\frac{D}{D-2s}$, and it deserves a detailed explanation. If we first take a look at the Kolmogorov example $\A_0$ in $\R^{2n}$ recalled in \eqref{kolmo}, the constant $D=4n$ is determined by the direct computation $V(t)=\omega_{2n}12^{-\frac{n}{2}}t^{2n}$. More generally, the fact that one can uniquely identify the constant $D$ in \eqref{vol} is shared by the subclass of \eqref{A0} which possesses invariances with respect to a family of non-isotropic dilations since in this case $V(t)\equiv V(1)t^{\frac{D}{2}}$ (to fix the ideas, for $\A_0$ the dilations are defined by $\delta_\lambda(v,x)=(\lambda v, \lambda^3 x)$): such homogeneous class was introduced in \cite{LP} where Lanconelli and Polidoro provided an explicit characterization in terms of the matrices $Q$ and $B$. Besides the homogeneous case, for the general class \eqref{A0} one can always say
\begin{equation}\label{D0}
\exists D_0\geq N\,\, \mbox{ such that }\,\, V(t)\ \cong\ t^{D_0/2}\,\, \mbox{ as  }t\to 0^+,
\end{equation}
see \cite[Section 2.4]{GThls}. We refer to the number $D_0$ as the intrinsic dimension of the semigroup $\{P_t\}_{t>0}$ at zero. With this perspective in mind, the validity of the assumption \eqref{vol} readily implies that $D\geq D_0$ (which in particular ensures $D>2s$ and $\frac{D}{D-2s}>1$) and it underscores the importance of the behaviour of $V(t)$ for large $t$. For example, the operator
$$
\Delta_v   + \left\langle v,\nabla_v \right\rangle + \left\langle v,\nabla_x \right\rangle\qquad\mbox{ in }\R^{2n}
$$
fits the assumptions of Theorem \ref{T:alt1} for every $D\geq 4n$: as a matter of fact one can check that in this case $V(t)=\omega_{2n}\left(2e^{t} - \frac{t}{2}-1+\frac{t}{2}e^{2t}-e^{2t}\right)^{n}$ which tells that $D_0=4n$ and $V(t)$ grows exponentially fast at $t\to \infty$. In \cite[Definition 3.4]{GThls} we introduced the notion of intrinsic dimension at infinity of the semigroup $\{P_t\}_{t>0}$ in order to handle the behaviour for large $t$ of the volume function. We are going to recall such notion in Definition \ref{D:hld} below, and we denote by $D_\infty$ such a constant. It turns out that $D_\infty \in [2,\infty]$ under assumption \eqref{trace}, and the validity of \eqref{vol} forces $D_\infty\geq D_0$ and $D\in [D_0,D_\infty]$. This gives a way to understand the operators $\A$ which do not satisfy the volume growth condition \eqref{vol}. In fact, we can consider the following operator
$$
\de^2_v  + v\de_x  - x\de_v\qquad\mbox{ in }\R^{2}
$$
for which $V(t)=\pi\left(\frac{t^2}{4}+\frac{1}{8}\left(\cos(2t)-1\right)\right)^{\frac{1}{2}}$: in this situation we thus have $D_0=4>D_\infty=2$ and therefore Theorem \ref{T:alt1} cannot apply to such case. In \cite[Theorem 1.2 and Theorem 7.6]{GTjfa} we established the following substitute result to treat the operators \eqref{A0} with $D_0> D_\infty$.

\begin{theorem}\label{T:alt2}
Let $0<s< \frac{1}{2}$. Assume \eqref{trace}, and $D_0> D_\infty$. Suppose there exists $\gamma>0$ such that
\begin{equation}\label{vol2}
V(t) \geq \gamma \min\{ t^{D_0/2},t^{D_\infty/2} \}.\quad\mbox{ for all }t>0.
\end{equation}
Then
$$\mathfrak B^{\A}_{2s,1}\hookrightarrow L^{\frac{D_0}{D_0-2s}}+L^{\frac{D_\infty}{D_\infty-2s}}.$$ 
More precisely, there exists a positive constant $c$ (depending on $N,D_0,D_\infty,s,\gamma$) such that for every $f\in \mathfrak B^{\A}_{2s,1}$ one has
\begin{equation}\label{besdueembe}
||f||_{L^{\frac{D_0}{D_0-2s}}+L^{\frac{D_\infty}{D_\infty-2s}}} \leq \frac{1}{c} \mathscr N^{\A}_{2s,1}(f).
\end{equation}
Moreover, there exists a positive constant $\tilde{c}$ (depending on $N,D_0,D_\infty,s,\gamma$) such that for any measurable set $E\subset \RN$ with $|E|<\infty$, one has
$$
\mathfrak P_{\A , s}(E) \geq\ \tilde{c} \min\left\{|E|^{\frac{D_0-2s}{D_0}}, |E|^{\frac{D_\infty-2s}{D_\infty}}\right\}.
$$
\end{theorem}
We notice that, under the assumptions of Theorem \ref{T:alt2}, we have $D_0> D_\infty>2s$ and then $\frac{D_\infty}{D_\infty-2s}>\frac{D_0}{D_0-2s}>1$. The proofs of Theorem \ref{T:alt1} and Theorem \ref{T:alt2} which can be found in \cite{GTjfa} were inspired by the powerful and flexible semigroup approach to isoperimetric inequalities which is due to Ledoux \cite{Led} in the local case. We also refer the interested reader to \cite{Pre, MPPP} for more insights on such a heat-kernel approach to perimeters and isoperimetric properties. The purpose of the present note is to provide a different proof of Theorems \ref{T:alt1} and \ref{T:alt2} which instead relies on an embedding in a weak $L^q$-space 
of a fractional $(p=1)$-Sobolev space: these Sobolev spaces are tailored on the fractional powers of $\A$ and their relevant embeddings were established in \cite{GThls}. For a proper historical perspective concerning the classical Sobolev spaces, we recall that from the representation formula $|f(X)|\le C(N) \int_{\RN} \frac{|\nabla f(Y)|}{|X-Y|^{N-1}} dY$, and the $L^1$-mapping properties of the Riesz potentials, one knows that $W^{1,1} \hookrightarrow L^{N/(N-1),\infty}$. A remarkable aspect of the end-point case $p=1$ is that such weak Sobolev embedding in fact implies the classical isoperimetric inequality $P(E) \ge C_N |E|^{\frac{N-1}N}$. The latter, in turn, combined with the coarea formula, is equivalent to the strong embedding $W^{1,1}\hookrightarrow L^{N/(N-1)}$. This establishes the beautiful fact that, in the geometric case $p=1$, the weak Sobolev embedding is equivalent to the strong one, and they are both equivalent to the isoperimetric inequality, see \cite{Ma2}. The main focus of this paper is a semigroup generalisation of this circle of ideas to the nonlocal degenerate setting of the operators $\A$.

\section{Preliminaries}\label{S:Sob}

In this section we recall the main ingredients we shall need in the proofs of Theorems \ref{T:alt1} and \ref{T:alt2} that we present in Section \ref{S:proof}. The main character in our analysis is the kernel in \eqref{PtKt}. As we mentioned it is well-known that $p(X,Y,t)$ is the heat kernel, i.e. that $p(\cdot, Y,\cdot)$ is a solution of the heat equation $\A u=\de_t u$ in $\RN\times (0,\infty)$ for any $Y\in\RN$ and $p(X,\cdot,t)$ tends to the Dirac delta $\delta_X$ in the distributional sense as $t\to 0^+$ for any $X\in\R^N$. In particular we are going to exploit the following property concerning such an approximation of the identity
\begin{equation}\label{delta}
P_t \varphi(X) \underset{t\to 0^+}{\longrightarrow} \varphi(X)\qquad \mbox{ for every }\phi\in L^\infty\cap\, C\left(\R^N\right) \mbox{ and }X\in\R^N.
\end{equation}
The limit in \eqref{delta} can be verified directly using the Markovian condition $P_t 1\equiv 1$ and the properties of the positive definite matrices $K(t)$ in \eqref{Kt} (see, e.g., \cite[Proposition 2.1]{GTker}; see also the analytic and probabilistic tools in \cite{DFP, Men} for a treatment of a more general class of operators with varying coefficients).\\
It is very convenient for us to exploit the notations we adopted with the explicit expression \eqref{PtKt}, as in this way we can put the $1$-parameter family of pseudoballs $B_t(X,\sqrt{t})$ at the center stage together with their volume function $V(t)$ in \eqref{defs}. We already stressed that, as a by product of the analysis in \cite{LP}, the small-time behaviour of $V(t)$ is governed by a suitable infinitesimal homogeneous structure which we encode in the number $D_0$ defined via \eqref{D0} (i.e. the intrinsic dimension of the semigroup $\{P_t\}_{t>0}$ at zero). On the other hand, the following definition allows us to handle the large-time behaviour of $V(t)$.

\begin{definition}\label{D:hld}
Consider the set 
$$\Sigma_\infty = \left\{\alpha>0\,\big|\, \int_1^\infty \frac{t^{\frac{\alpha}{2}-1}}{V(t)} dt < \infty\right\}.$$
We call the number  
$$D_\infty = \sup \Sigma_\infty$$
the \emph{intrinsic dimension at infinity} of the semigroup $\{P_t\}_{t>0}$. 
\end{definition}
Thanks to the study of the large-time behaviour of the eigenvalues of $tK(t)$ performed in \cite[Section 3]{GThls} (see also \cite[Proposition 2.3]{BGT}), we know that $2\leq D_\infty\leq \infty$ under the assumption \eqref{trace}.

\subsection{Fractional powers of $\mathscr A$, and their Sobolev embeddings}\label{sub1}
In \cite{GTind} we developed a fractional calculus for $\A$. On functions belonging to the Schwartz class $\So$ the nonlocal operator $\As$, for $0<s<1$, is defined through the following pointwise formula
\begin{equation}\label{defAs}
(-\mathscr A)^s f(X) =  - \frac{s}{\G(1-s)} \int_0^\infty t^{-(1+s)} \left[P_t f(X) - f(X)\right] dt,\qquad X\in\RN.
\end{equation}
If $\operatorname{tr} B \ge 0$, the formula in \eqref{defAs} defines in fact an $L^p$-function for any $1\leq p<\infty$. We can then extend the operator $\As$ to a closed operator on its domain in $L^p$ endowed with the graph norm: this is precisely what we are doing with the following definition.
\begin{definition}
Let $1 \leq p <\infty$, $0<s<1$, and assume \eqref{trace}. We define the Sobolev space $\Lo$ as
$$
\Lo = \left\{f\in L^p\mid \As f \in L^p\right\}
$$
with
$$
\|f\|_{\Lo}=\|f\|_p+\|\As f\|_p.
$$
\end{definition}
Thanks to a density result the Banach space $\Lo$ coincides with the completion of $\So$ with respect to $\|\cdot\|_{\Lo}$, see in this respect \cite[Definition 4.4 and Proposition 4.6]{GThls} as well as \cite[Proposition 2.13]{GTjfa}.\\
In \cite[Theorem 7.5 and Theorem 7.7]{GThls} we proved the following

\begin{theorem}\label{T:sob}
Let $0<s< 1$, and assume \eqref{trace}.
\begin{itemize}
\item[(a)] If \eqref{vol} hold, then we have $\mathscr L^{2s,1}\ \hookrightarrow\ L^{\frac{D}{D-2s},\infty}$. More precisely, there exists a constant $S_{1,s} >0$, depending on $N,D,s,\gamma_D$, such that for any $f\in \mathscr L^{2s,1}$ one has
\begin{equation}\label{weak11}
\underset{\la>0}{\sup}\ \la |\{X\in \RN\mid |f(X)| > \la\}|^{\frac{D-2s}{D}} \le S_{1,s} \|\As f\|_{1}.
\end{equation}
\item[(b)] If instead $D_0> D_\infty$ and \eqref{vol2} hold, we have $\mathscr L^{2s,1}\ \hookrightarrow\ L^{\frac{D_0}{D_0-2s},\infty} + L^{\frac{D_\infty}{D_\infty-2s},\infty}$. More precisely, there exists a constant $S_{1,s} >0$, depending on $N,D_\infty, D_0, s,\gamma$, such that for any $f\in \mathscr L^{2s,1}$ one has
\begin{align}\label{weak12}
\min\{\underset{\la>0}{\sup}\ \la\ |\{X\mid |f(X)|>\la\}|^{\frac{D_0-2s}{D_0}}, \underset{\la>0}{\sup}\ \la\ |\{X\mid | f(X)|>\la\}|^{ \frac{D_\infty -2s}{D_\infty}} \}& \nonumber\\
\leq S_{1,s} \|\As f\|_{1}.&
\end{align}
\end{itemize}
\end{theorem}
We refer the reader to \cite{GThls} for the case $p>1$, where we established the strong embeddings
\begin{equation}\label{semb}
\Lo\hookrightarrow\ L^{\frac{pD}{D-2s p}}\qquad\mbox{ and }\qquad \Lo\hookrightarrow L^{\frac{p D_0}{D_0-2s p}} + L^{\frac{p D_\infty}{D_\infty-2s p}}
\end{equation}
under the respective assumptions \eqref{vol} and \eqref{vol2} (with $p<\frac{D}{2s}$ and $p<\frac{D_\infty}{2s}$).\\
Let us spend some words on the proofs provided in \cite{GThls} of Theorem \ref{T:sob} and \eqref{semb}, which were inspired by the works by Varopoulos in the framework of positive symmetric semigroups (see, e.g., \cite{V85}). With the aid of a crucial inversion formula for the fractional powers of $\A$ in terms of suitable Riesz-type potentials having the following semigroup representation
$$
f\mapsto \frac{1}{\G(s)} \int_0^\infty t^{s-1} P_t f dt,
$$
we were able to deduce the proof of our Sobolev-type embeddings from the $L^p-L^q$ mapping properties of these Riesz potentials. The identification of such $q$ is the point in the proof where the volume growth conditions \eqref{vol} and \eqref{vol2} come into play. The key technical tool to show these mapping properties is the introduction of a maximal function related to $\A$, which we believe has interest in its own. Influenced by the powerful ideas by Stein in \cite{Steinlp}, we exploited the Poisson semigroup $e^{z\sqrt{-\A}}$ to define the maximal function
$$
\Ma^\star f(X) = \underset{z>0}{\sup}\ \left|\frac{1}{\sqrt{4\pi}} \int_0^\infty \frac{z}{t^{3/2}} e^{-\frac{z^2}{4t}} P_t f(X) dt\right|,\ \ \ \ X\in \RN.
$$
The operator $\Ma^\star$ maps in fact continuously $L^1$ in $L^{1,\infty}$ and any $L^p$ in itself for $p>1$ (see \cite[Theorem 5.5]{GThls}), i.e. the following maximal theorem for the class \eqref{A0} holds true.
\begin{theorem}\label{T:maximal}
Assume \eqref{trace}.
\begin{itemize}
\item[(i)] There exists a universal constant $A_1>0$ such that, given $f\in L^1$, one has
\[
\underset{\la>0}{\sup}\ \la |\{X\in \RN \mid \Ma^\star f(X) > \la\}| \le A_1 \|f\|_{1};
\]
\item[(ii)] If $1<p\le \infty$, there exists a universal constant $A_p>0$ such that for any $f\in L^p$ one has
\[
\|\Ma^\star f\|_{p} \le A_p \|f\|_{p}.
\]
\end{itemize}
\end{theorem}

\subsection{Nonlocal perimeters and coarea formulas}\label{sub2} As we want to go back to the study of the Besov spaces $\mathfrak B^\sA_{s,p}$ and their seminorm $\mathscr N^{\A}_{s,p}$ defined in \eqref{sn}, we start by recalling the relationship between the $s$-perimeter associated to $\sA$ and the fractional power $\As$. Keeping in mind Definition \ref{defper} and \eqref{defAs}, we have in fact
\begin{equation}\label{confr}
\mathfrak P_{\A, s}(E)=\frac{\G(1-s)}{s} \|\As \In\|_1\qquad \mbox{ if }\In \in \mathfrak B^{\A}_{2s,1}.
\end{equation}
We refer to \cite[Corollary 3.6]{GTfi} and \cite[Lemma 3.3]{GTjfa} for a proof of \eqref{confr}. Moreover, in \cite[Proposition 3.3]{GTfi} we showed the boundedness of the map $\As: \mathfrak B^{\A}_{2s,1}\longrightarrow L^1$. This says that 
\begin{equation}\label{conbound} 
\mathfrak B^{\A}_{2s,1}\hookrightarrow \mathscr L^{2s,1}.
\end{equation}
Like the Sobolev-type spaces $\Lo$, also the Besov spaces $\mathfrak B^\sA_{s,p}$ enjoy useful density properties. In the next section it will play a role the following density result which was established in \cite[Proposition 3.2]{BGT} (see also \cite[Lemma 7.3]{GTjfa}).
\begin{proposition}\label{dens}
Assume \eqref{trace}. For every $0<s<1$ and $1\leq p<\infty$, we have 
$$\overline{C^\infty_0}^{\mathfrak B^\sA_{s,p}}=\mathfrak B^\sA_{s,p}.$$
\end{proposition}
Another important tool for our puroposes is the validity of a coarea formula which yields a further link between the seminorm $\mathscr N^{\A}_{2s,1}$ and the nonlocal perimeter $\mathfrak P_{\A, s}$.
\begin{proposition}\label{coar}
Let $0<s<\frac{1}{2}$, and assume \eqref{trace}. For any $f\in \mathfrak B^{\A}_{2s,1}$ we have
\begin{equation}\label{chiaveinstrong}
\mathscr N^{\A}_{2s,1}(f)= \int_\R \mathfrak P_{\A, s}\left(\{f>\sigma\}\right) d\sigma.
\end{equation}
\end{proposition}
\begin{proof}
It is a consequence of the results in \cite[Proposition 7.4 and formula (7.4)]{GTjfa}, once we keep in mind \eqref{confr} and \eqref{weighthejump}.
\end{proof}

\begin{remark} If one compares Definition \ref{defper} with the definition of $s$-perimeter associated to $\sA$ given in \cite[Section 4.1]{GTjfa}, one can notice a difference that we now want to comment on. Via a relaxation procedure, in \cite{GTjfa} we denoted by $\mathfrak P^{\A}_{s}(E)$ the following
$$
{\inf}\left\{ \underset{k\to \infty}{\liminf}\ \|\As f_k\|_1 \,\mid\, \{f_k\}_{k\in \mathbb N}\in\So \mbox{ such that }f_k \to \In\mbox{ in }L^1  \right\}
$$
for $E\subset \RN$ measurable and with finite volume. Thanks to Proposition \ref{dens} it is clear that, if $\In \in \mathfrak B^{\A}_{2s,1}$, then there exists $\{f_k\}_{k\in \mathbb N}\in C_0^\infty \subset \So$ converging to $\In$ in $\mathfrak B^{\A}_{2s,1}$. For such a sequence we then obtain
$$
\underset{k\to \infty}{\lim}\ \|\As f_k\|_1 = \|\As \In\|_1= \frac{s}{\G(1-s)}  \mathfrak P_{\A, s}(E),
$$
where we exploited \eqref{conbound} and \eqref{confr}. Hence we have
$$
\mathfrak P^{\A}_{s}(E)\leq \frac{s}{\G(1-s)}  \mathfrak P_{\A, s}(E)\quad\mbox{ in case }\In \in \mathfrak B^{\A}_{2s,1}.
$$
\end{remark}

\begin{remark} In \cite[Section 4.2]{GTjfa} we introduced another notion of $s$-perimeter which we denoted by $\mathfrak P^{\A, \star}_{s}(E)$. For bounded measurable sets $E\subset \RN$, we let
$$
\mathfrak P^{\A, \star}_{s}(E)=\underset{t\to 0^+}{\lim}\ \|\As P_t \In\|_1.
$$
The previous definition makes sense as $P_t \In$ belongs to $\So$ for any $t>0$ and $t\mapsto \|\As P_t \In\|_1$ is monotone decreasing (see \cite[pg. 21]{GTjfa}). Keeping in mind \eqref{conbound}, we have by \cite[Corollary 3.5]{GTjfa} and \eqref{confr} that
$$
\mathfrak P^{\A, \star}_{s}(E)= \|\As \In\|_1 = \frac{s}{\G(1-s)}  \mathfrak P_{\A, s}(E)\quad\mbox{ in case }\In \in \mathfrak B^{\A}_{2s,1}.
$$
\end{remark}

From \eqref{conbound} and Theorem \ref{T:sob}, we immediately deduce the validity of the embeddings
\begin{equation}\label{weak}
\mathfrak B^{\A}_{2s,1}  \hookrightarrow L^{\frac{D}{D-2s},\infty} \qquad\mbox{and}\qquad \mathfrak B^{\A}_{2s,1}\hookrightarrow L^{\frac{D_0}{D_0-2s},\infty} + L^{\frac{D_\infty}{D_\infty-2s},\infty}
\end{equation}
under the respective volume growth conditions \eqref{vol} and \eqref{vol2}. The objective of the next section is to replace $L^{q,\infty}$ with the strong spaces $L^{q}$ in \eqref{weak}.

\section{Proofs}\label{S:proof}

The proof of Theorem \ref{T:alt1} and Theorem \ref{T:alt2} will result as a combination of Theorem \ref{T:sob} with the following lemma.

\begin{lemma}\label{Lemma:interiorblowup}
Consider a measurable set $E\subset \RN$ with $|E|<\infty$. Then we have
$$
\underset{t\to 0^+}{\liminf} \left| \left\{ X\in\RN\mid P_t\In (X)>\frac{1}{2}\right\}\right|
\geq \left| {\rm{int}} E \right|
$$
\end{lemma}
\begin{proof}
We start by noticing that
\begin{equation}\label{blowupclaim}
\mbox{ for every }X\in {\rm{int}} E \mbox{ we have }\underset{t\to 0^+}{\lim} P_t\In (X) =1.
\end{equation}
The statement in \eqref{blowupclaim} is a consequence of \eqref{delta}. As a matter of fact, for any $x \in {\rm{int}} E$, we can pick $\rho_X>0$ such that an open neighborhood of size $\rho_X$ is contained in $E$. Thanks to this fact, it is easy to construct a continuous function $\varphi_X$ such that 
$$0\leq \varphi_X\leq 1, \quad \varphi_X(X)=1,\quad \varphi_X\equiv 0 \mbox{ in }\RN\smallsetminus E.$$
Hence we have
$$
1\geq P_t\In (X) \geq P_t \varphi_X(X) \underset{t\to 0^+}{\longrightarrow} \varphi_X(X)=1,
$$
which ensures the validity of \eqref{blowupclaim}. Therefore
$$
\mbox{ for every }X\in {\rm{int}} E \,\,\mbox{ there exists }t_X>0\,\,\mbox{ such that } P_t\In (X)>\frac{1}{2} \,\,\mbox{ for }0<t<t_X.
$$
This fact implies that
$$
\underset{t\to 0^+}{\lim} \mathbf 1_{E_t}(X) =1 \quad\mbox{ for every }x\in {\rm{int}} E
$$
once we denote
$$
E_t=\left\{ X\in\RN\mid P_t\In (X)>\frac{1}{2}\right\}.
$$
By Fatou's Lemma we then obtain
$$
\underset{t\to 0^+}{\liminf}\int_{\RN} \mathbf 1_{E_t}(X)  dX \geq \int_{\RN} \left(\underset{t\to 0^+}{\liminf} \mathbf 1_{E_t}(X) \right)  dX \geq \int_{{\rm{int}} E} 1 dX,
$$
which completes the proof of the desired statement.
\end{proof}

We are thus ready to provide the proof of Theorem \ref{T:alt1}.

\begin{proof}[Proof of Theorem \ref{T:alt1}]
For any $f\in \mathfrak B^{\A}_{2s,1}$, we denote
$$E_\sigma=\{X\in\RN\mid |f(X)|>\sigma\},\qquad\mbox{for }\sigma>0.$$
Since by Proposition \ref{coar} we have
$$\int_{0}^{\infty} \mathfrak P_{\A, s}\left(E_\sigma\right) d\sigma = \mathscr N^{\A}_{2s,1}(|f|)\leq \mathscr N^{\A}_{2s,1}(f)<\infty,$$
it is clear that for almost any $\sigma$ one has $1_{E_\sigma}\in \mathfrak B^{\A}_{2s,1}$. As a consequence, for such $\sigma$, we obtain from \cite[Corollary 3.5]{GTjfa} and \eqref{confr} that
\begin{equation}\label{onelimit}
\underset{t\to 0^+}{\lim}\|\As P_t 1_{E_\sigma}\|_1=\|\As 1_{E_\sigma}\|_1=\frac{s}{\G(1-s)} \mathfrak P_{\A, s}(E_\sigma).
\end{equation}
The aim of the proof is to establish the following bound
\begin{equation}\label{smooth}
||f||_{\frac{D}{D-2s}} \leq \frac{2s S_{1,s}}{\G(1-s)} \mathscr N^{\A}_{2s,1}(f) \quad\mbox{ for every }f\in C^\infty\cap \mathfrak B^{\A}_{2s,1},
\end{equation}
where $S_{1,s}$ is the positive constant appearing in \eqref{weak11}. Thus, let us fix $f\in C^\infty\cap \mathfrak B^{\A}_{2s,1}$ and denote $G(\sigma)=|E_\sigma|$. Since $E_\sigma$ has finite measure (by Chebyshev's inequality) and it is open (by the continuity of $f$), we are entitled to apply Lemma \ref{Lemma:interiorblowup} and we obtain
\begin{equation}\label{fromlemma}
G(\sigma)\leq \underset{t\to 0^+}{\liminf} \left| \left\{ X\in\RN\mid P_t\mathbf 1_{E_\sigma} (X)>\frac{1}{2}\right\}\right|.
\end{equation}
On the other hand, since for $t>0$ the positive function $P_t 1_{E_\sigma}$ belongs to $\mathscr L^{2s,1}$ (this holds true for any $\sigma$ such that $1_{E_\sigma}\in \mathfrak B^{\A}_{2s,1}$, see in this respect \cite[Lemma 3.4]{GTjfa}), Theorem \ref{T:sob} yields
\begin{equation}\label{conseqsob}
\frac{1}{2}\left| \left\{ X\in\RN\mid P_t\mathbf 1_{E_\sigma} (X)>\frac{1}{2}\right\}\right|^{\frac{D-2s}{D}} \le S_{1,s} ||\As P_t\mathbf 1_{E_\sigma}||_{1}.
\end{equation}
Therefore, from the combination of \eqref{fromlemma}, \eqref{conseqsob}, and \eqref{onelimit}, we obtain
$$
G(\sigma)^{\frac{D-2s}D} \leq \frac{2s S_{1,s}}{\G(1-s)}\mathfrak P_{\A, s}(E_\sigma)
$$
Since $G$ is non-increasing and $D\geq D_0\ge 2 >2s$, we then have
\begin{align*}
||f||_{\frac{D}{D-2s}}&=\left(\int_{\RN} |f|^\frac{D}{D-2s}(X) dX\right)^\frac{D-2s}{D} = \left(\frac{D}{D-2s} \int_0^\infty \sigma^{\frac{2s}{D-2s}} G(\sigma) d\sigma\right)^\frac{D-2s}{D}\\
&\leq \int_0^\infty G(\sigma)^{\frac{D-2s}D} d\sigma\leq \frac{2s S_{1,s}}{\G(1-s)}\int_{0}^{\infty} \mathfrak P_{\A, s}\left(E_\sigma\right) d\sigma =\frac{2s S_{1,s}}{\G(1-s)} \mathscr N^{\A}_{2s,1}(|f|)\\
&\leq \frac{2s S_{1,s}}{\G(1-s)} \mathscr N^{\A}_{2s,1}(f).
\end{align*}
The previous inequality proves the desired \eqref{smooth}. Hence, the density of $C^\infty$ in $\mathfrak B^{\A}_{2s,1}$ (which is a consequence of Proposition \ref{dens}) implies the validity of \eqref{besunoembe} with the choice
$$
c=\frac{\G(1-s)}{2s S_{1,s}}.
$$
In particular, if $E\subset \RN$ has finite $s$-perimeter, we can plug $f=\In$ in \eqref{besunoembe} and we deduce
$$
\mathfrak P_{\A, s}(E)=\mathscr N^{\A}_{2s,1}(\In)\geq \frac{\G(1-s)}{2s S_{1,s}}\|\In\|_{\frac{D}{D-2s}}= \frac{\G(1-s)}{2s S_{1,s}}|E|^{\frac{D-2s}{D}},
$$
which completes the proof of the theorem.
\end{proof}

We conclude the paper with the proof of Theorem \ref{T:alt2}. To this aim we fix some notation. We set
$$q_0=\frac{D_0}{D_0-2s}\quad\mbox{ and }\quad q_\infty=\frac{D_\infty}{D_\infty-2s}.$$
We recall that, in the assumptions of the theorem, we have $D_0>D_\infty\geq 2>2s$. This says in particular that
\begin{equation}\label{order}
q_\infty>q_0>1.
\end{equation}
Let us also recall that, when we write $L^{q_0}+L^{q_\infty}$, we mean the Banach space of functions $f$ which can be written as $f=f_1+f_2$ with $f_1\in L^{q_0}$ and $f_2\in L^{q_\infty}$, endowed with the norm
$$||f||_{L^{q_0}+L^{q_\infty}} = \inf_{f=f_1+f_2,\\ f_1\in L^{q_0},\,f_2\in L^{q_\infty}}{||f_1||_{q_0}+||f_2||_{q_\infty}}.$$

\begin{proof}[Proof of Theorem \ref{T:alt2}]
We want to argue as similar as possible to the proof of Theorem \ref{T:alt1}, from which we also borrow the notations for the sets $E_\sigma$ and the non-increasing function $G(\sigma)$. We then fix an arbitrary function $f\in C^\infty\cap \mathfrak B^{\A}_{2s,1}$. From the combination of Lemma \ref{Lemma:interiorblowup}, Theorem \ref{T:sob}, and \eqref{onelimit}, we obtain
$$
\min\left\{G(\sigma)^{\frac{1}{q_0}}, G(\sigma)^{\frac{1}{q_\infty}} \right\}\leq \frac{2s S_{1,s}}{\G(1-s)}\mathfrak P_{\A, s}(E_\sigma),
$$
where $S_{1,s}$ is the positive constant appearing in \eqref{weak12}. If we then exploit the coarea formula \eqref{chiaveinstrong}, we deduce that
\begin{equation}\label{quasimb}
\int_0^\infty \min\left\{G(\sigma)^{\frac{1}{q_0}}, G(\sigma)^{\frac{1}{q_\infty}} \right\} d\sigma \leq \frac{2s S_{1,s}}{\G(1-s)} \mathscr N^{\A}_{2s,1}(|f|)\leq \frac{2s S_{1,s}}{\G(1-s)} \mathscr N^{\A}_{2s,1}(f).
\end{equation}
We now want to provide a lower bound for the left-hand side of \eqref{quasimb} in terms of $||f||_{L^{q_0}+L^{q_\infty}}$. To this aim, we denote $$\sigma_f=\sup\{\sigma> 0\,:\, G(\sigma)> 1\}.$$ If $|E_\sigma|\leq 1$ for all $\sigma$, we agree to let $\sigma_f=0$. Since $f\in L^1$ we have that $\sigma_f\in [0,\infty)$. We consider
\begin{equation}\label{funodue}
f_1(X)=f(X)\mathbf 1_{E_{\sigma_f}}(X)\quad\mbox{and}\quad f_2(X)=f(X)(1-\mathbf 1_{E_{\sigma_f}}(X)).
\end{equation}
We notice that $f_1(X)$ and $f_2(X)$ cannot be both non-null for the same $X$, and in particular the following holds true
$$
\left|f(X)\right|=\left|f_1(X) + f_2(X)\right|=\left|f_1(X)\right| + \left|f_2(X)\right|.
$$
We also make use of the notation $E^i_\sigma=\{X\in\RN\,:\,\left|f_i(X)\right|>\sigma\}$ for $i\in \{1,2\}$. One can check the following relations (see also \cite[pg. 38]{GTjfa})
$$E^1_\sigma=\begin{cases}
E_\sigma \qquad\,\,\mbox{if }\sigma>\sigma_f,  \\
E_{\sigma_f} \qquad \mbox{if }\sigma\leq\sigma_f
\end{cases}\quad\mbox{ and }\quad E^2_\sigma=\begin{cases}
\varnothing \,\,\,\quad\quad\quad\qquad\mbox{if }\sigma>\sigma_f,  \\
E_\sigma\smallsetminus E_{\sigma_f} \qquad \mbox{if }\sigma\leq\sigma_f.
\end{cases}$$
Since $|E^1_\sigma|\leq 1$ and $|E^1_\sigma|\leq G(\sigma)$ for all $\sigma$, by \eqref{order} we obtain
\begin{align}\label{funo}
||f_1||_{q_0}&=\left(q_0 \int_0^\infty \sigma^{q_0-1} |E^1_\sigma| d\sigma\right)^\frac{1}{q_0}\leq \int_0^\infty |E^1_\sigma|^{\frac{1}{q_0}} d\sigma = \int_0^\infty \min\left\{|E^1_\sigma|^{\frac{1}{q_0}},|E^1_\sigma|^{\frac{1}{q_\infty}}\right\} d\sigma\\
&\leq \int_0^\infty \min\left\{G(\sigma)^{\frac{1}{q_0}}, G(\sigma)^{\frac{1}{q_\infty}} \right\} d\sigma.\nonumber
\end{align}
On the other hand, since $G(\sigma)$ is bigger than $1$ on the interval $(0,\sigma_f)$, by \eqref{order} we also have
\begin{align}\label{fdue}
||f_2||_{q_\infty}&= \left(q_\infty \int_0^\infty \sigma^{q_\infty-1} |E^2_\sigma| d\sigma\right)^\frac{1}{q_\infty}\leq \int_0^\infty |E^2_\sigma|^{\frac{1}{q_\infty}} d\sigma= \int_0^{\sigma_f} |E^2_\sigma|^{\frac{D_\infty-2s}{D_\infty}} d\sigma\\
&\leq \int_0^{\sigma_f} G(\sigma)^{\frac{1}{q_\infty}} d\sigma=\int_0^{\sigma_f} \min\left\{G(\sigma)^{\frac{1}{q_0}}, G(\sigma)^{\frac{1}{q_\infty}} \right\} d\sigma.\nonumber
\end{align}
The combination of \eqref{funo} and \eqref{fdue} yields
$$||f_1||_{q_0}+||f_2||_{q_\infty}\leq 2\int_0^\infty \min\left\{G(\sigma)^{\frac{1}{q_0}}, G(\sigma)^{\frac{1}{q_\infty}} \right\} d\sigma,$$
Keeping in mind \eqref{quasimb}, we have just proved that
\begin{equation}\label{eccolaa}
||f||_{L^{q_0}+L^{q_\infty}}\leq  ||f_1||_{q_0}+||f_2||_{q_\infty}\leq \frac{4s S_{1,s}}{\G(1-s)} \mathscr N^{\A}_{2s,1}(f)\quad\mbox{ for every }f\in C^\infty\cap \mathfrak B^{\A}_{2s,1}.
\end{equation}
The density of $C^\infty$ in $\mathfrak B^{\A}_{2s,1}$ implies then the validity of \eqref{besdueembe} with the choice
$$
c=\frac{\G(1-s)}{4s S_{1,s}}.
$$
Finally, since it is not completely obvious to deduce from \eqref{besdueembe} the nonlocal isoperimetric inequality stated in Theorem \ref{T:alt2}, we provide here the details. Let us thus take a set $E\subset \RN$ with finite $s$-perimeter. We can consider a sequence $\{f_k\}_{k\in\N}$ of Friedrichs' mollifiers for the function $\In$. We recall that
$$
f_k\in  C^\infty\cap \mathfrak B^{\A}_{2s,1}, 	\qquad\qquad 0\leq f_k\leq 1. 
$$
If we let $(f_k)_1$ and $(f_k)_2$ the splitting of the function $f_k$ according to \eqref{funodue}, since $0\leq (f_k)_1,(f_k)_2\leq 1$ we have
\begin{align}\label{weirdconstant}
&\min\left\{\left(\int_{\RN} |(f_k)_1(X)|^{q_0} dX + \int_{\RN} |(f_k)_2(X)|^{q_\infty} dX\right)^{\frac{1}{q_0}},\right.\\
&\hspace{1cm}\left. \left(\int_{\RN} |(f_k)_1(X)|^{q_0} dX + \int_{\RN} |(f_k)_2(X)|^{q_\infty} dX\right)^{\frac{1}{q_\infty}} \right\}\nonumber\\
&\geq \min\left\{\left(\int_{\RN} |(f_k)_1(X)|^{q_\infty} dX + \int_{\RN} |(f_k)_2(X)|^{q_\infty} dX\right)^{\frac{1}{q_0}},\right.\nonumber\\
&\hspace{1cm}\left. \left(\int_{\RN} |(f_k)_1(X)|^{q_\infty} dX + \int_{\RN} |(f_k)_2(X)|^{q_\infty} dX\right)^{\frac{1}{q_\infty}} \right\}\nonumber\\
&=\min\left\{\|f_k\|^{\frac{q_\infty}{q_0}}_{q_\infty}, \|f_k\|_{q_\infty} \right\},\nonumber
\end{align}
where we used \eqref{order} and the fact that $(f_k)_1(f_k)_2\equiv 0$. On the other hand, denoting $c(q_0,q_\infty)=\min\{x^{\frac{1}{q_0}}+y^{\frac{1}{q_\infty}}\,:\, x,y\geq 0,\mbox{ and }x+y=1\}$, we have the validity of the simple inequality 
$$x^{\frac{1}{q_0}}+y^{\frac{1}{q_\infty}}\geq c(q_0,q_\infty) \min\left\{(x+y)^{\frac{1}{q_0}}, (x+y)^{\frac{1}{q_\infty}} \right\}\qquad\mbox{ for all }x,y\geq 0,
$$
which implies
\begin{align}\label{dalbassounodue}
\|(f_k)_1\|_{q_0}+ \|(f_k)_2\|_{q_\infty} \geq& c(q_0,q_\infty) \min\left\{\left(\int_{\RN} |(f_k)_1(X)|^{q_0} dX + \int_{\RN} |(f_k)_2(X)|^{q_\infty} dX\right)^{\frac{1}{q_0}},\right.\\
&\hspace{2cm}\left. \left(\int_{\RN} |(f_k)_1(X)|^{q_0} dX + \int_{\RN} |(f_k)_2(X)|^{q_\infty} dX\right)^{\frac{1}{q_\infty}} \right\}.\nonumber
\end{align}
Putting together \eqref{eccolaa} with \eqref{weirdconstant}-\eqref{dalbassounodue}, we obtain
$$
\mathscr N^{\A}_{2s,1}(f_k) \geq \frac{\G(1-s)}{4s S_{1,s}}\left(\|(f_k)_1\|_{q_0}+ \|(f_k)_2\|_{q_\infty}\right) \geq \frac{\G(1-s) c(q_0,q_\infty)}{4s S_{1,s}} \min\left\{\|f_k\|^{\frac{q_\infty}{q_0}}_{q_\infty}, \|f_k\|_{q_\infty} \right\}
$$
for any $k\in\N$. By letting $k\to\infty$, since $f_k\to \In$ in every $L^p$-space and also in $\mathfrak B^{\A}_{2s,1}$ (see \cite[Proposition 3.2, Step I]{BGT}), we deduce the desired
$$
\mathfrak P^{\A}_{s}(E) \geq\ \frac{\G(1-s) c(q_0,q_\infty)}{4s S_{1,s}} \min\left\{|E|^{\frac{1}{q_0}}, |E|^{\frac{1}{q_\infty}}\right\}.
$$
\end{proof}

\section*{Acknowledgments}
\noindent We wish to thank St\'ephane Menozzi, Andrea Pascucci, and Sergio Polidoro for the organization and the kind invitation to the conference ``{\em Kolmogorov operators and their applications}" held in June 2022 in Cortona.

Both authors are supported in part by a Progetto SID: ``Aspects of nonlocal operators via fine properties of heat kernels", University of Padova, 2022. The first author has also been partially supported by a Visiting Professorship at the Arizona State University. The second author has been partially supported by the Gruppo Nazionale per l'Analisi Matematica, la Probabilit\`a e le loro Applicazioni (GNAMPA) of the Istituto Nazionale di Alta Matematica (INdAM).


\bibliographystyle{amsplain}

\end{document}